\documentclass[11pt]{article}
\usepackage{a4,amssymb,amsmath,amsthm,comment}
\begin{document}
\title{Bounded prime gaps in short intervals}
\theoremstyle{plain}
\newtheorem{thm}{Theorem} \renewcommand{\thethm}{\hskip -4  pt}
\newtheorem{lem}{Lemma} 
\newtheorem{cor}{Corollary}
\theoremstyle{definition} 
\newtheorem{example}{Example}
\newtheorem{prob}{Problem}\newtheorem{conj}{Conjecture}
\newtheorem{defn}{Definition}
\newtheorem{rem}{Remark}
\newtheorem{ack}{Acknowledgements}
\renewcommand{\theack}{\hskip -4 pt}
\def\cprime{$'$}
\newcommand{\C}{{\mathbb C}}\newcommand{\Hh}{{\mathcal H}} 
\newcommand{\R}{{\mathbb R}} \newcommand{\N}{{\mathbb N}}
\newcommand{\Z}{{\mathbb Z}}
\newcommand{\abs}[1]{{\left| {#1} \right|}} \newcommand{\p}[1]{{\left(
      {#1} \right)}} \newcommand{\jtf}[1]{{#1}^{\diamond}}

\newcommand{\Oh}[1]{{O \p{#1}}}

\renewcommand{\Re}{\operatorname{Re}} \renewcommand{\Im}{\operatorname{Im}} 
\nonstopmode
\date{}
 \author{Johan Andersson\thanks{Stockholm University, Department of Mathematics, Stockholm University, SE - 106 91 Stockholm, Sweden, Email: johana@math.su.se}}

\maketitle
\begin{abstract} We generalise  Zhang's and Pintz recent results on bounded prime gaps to give a lower bound for the the number of prime pairs bounded by $6 \cdot 10^7$ in the short interval $[x,x+x (\log x)^{-A}]$. Our result follows only by analysing Zhang's proof of Theorem 1, but we also explain how a sharper variant of Zhang's Theorem 2 would imply the same result for shorter intervals. 
\end{abstract}

Yitang Zhang \cite{Zhang} in his recent landmark paper proved the there are infinitely many weak prime pairs. More precisely he proved that
\begin{gather} \label{b}
  \liminf_{n \to \infty} (p_{n+1}-p_n)<7 \cdot 10^7.
\end{gather}
 His method is a variant of the method of Goldston-Pintz-Yildirim \cite{GPY1,GPY2}. For a nice introduction to this method see Soundararajan \cite{Sound}. Zhang not only realised that it is sufficient to consider divisors $d$ with only small prime factors in the definition of the real arithmetical function $\lambda(n)$ that is fundamental in the method\footnote{This had been realised already by Motohashi-Pintz \cite{MotPin}, see the remark in Pintz \cite{Pintz}, although they did not manage to prove what corresponds to Theorem 2 in Zhang \cite{Zhang} that ultimately is what was needed.}, but he also managed to utilise this insight. As in previous work on related results, Zhang used the notation of admissible set. A set \begin{gather} \Hh= \{h_1,\ldots,h_{k_0} \} \end{gather}  is called admissible if there is no obvious arithmetical obstruction that makes it impossible for $\Hh+n$ to be prime for all elements (such as $\Hh=\{1,2\}$) except for a possible exceptional case. The Hardy-Littlewood conjecture asserts
that the number of $n$ less than $x$ such that $n+\Hh$ consists of only primes should be asymptotic to $\mathfrak S(\Hh) x/(\log x)^k$, where $\mathfrak S(\Hh)$ is a certain positive constant.  While far from being proven, it gives good intuition what the right answer should be for these questions. 

Zhang's result \eqref{b}   follows from his result that for any admissible $\Hh$ with at least $3.5 \cdot 10^6$ elements there are infinitely many $n$ such that the set $n+\Hh$ contains at least two primes.  By being more careful with the estimates and choice of the admissible set, Trudgian \cite{Trudgian} already improved Zhang's constant $7 \cdot 10^7$ to $6 \cdot 10^7$. There is also a team effort led by Terence Tao that as of this moment has managed to improve the bound to a number slightly less than $4\cdot 10^5$; for the latest results see \cite{polymath}. Since this is work in progress and depends on more substantial changes to the proof of Zhang's Theorem 1, we have chosen not to use this result in this version of the paper.

 Zhang's bounded prime gap theorem ultimately follows from the inequality
\begin{gather} \label{a} \sum_{x \leq n \leq 2x}\left (\sum_{i=1}^{k_0} \theta(n+h_i) -\log(3x) \right) (\lambda(n))^2 \geq  (\omega+o(1))x (\log x)^{k_0+l_0+1}, 
\end{gather}
for some constant $\omega>0$ where 

\begin{gather*}
     \theta(n)=\begin{cases} \log n, & \text{if $n$ is prime,} \\ 0, & \text{otherwise,} \end{cases} \end{gather*} and  $\lambda$ is defined as follows (following Zhang \cite[(2.12)]{Zhang}):
\begin{gather}
 \label{lambdadef} \lambda(n)=\sum_{d|(P(n),\mathcal P)}  \mu(d) g(d), \qquad  g(y)=\begin{cases} \frac 1 {(k_0+l_0)!} \left(\log \frac D y \right)^{k_0+l_0}, & y<D, \\ 0, &   y \geq D, \end{cases}\\ \intertext{where} \notag
    k_0=3.5 \cdot 10^6, \qquad l_0=180, \qquad \varpi=1/1168, \qquad D=x^{\varpi+1/4},   \\
    D_1=x^\varpi,  \qquad \mathcal P=\prod_{p \leq D_1} p, \qquad P(n)=\prod_{i=1}^{k_0} (n+h_i). \notag
 \end{gather}
By estimating $\lambda(n)$ trivially by its definition  \eqref{lambdadef} and absolute values we see that \[\lambda(n) \ll \tau\left(\prod_{i=1}^{k_0} (n+h)\right) (\log D)^{k_0+l_0},\] where $\tau(n)$ denotes the divisor function, and from the well known estimate $\tau(n) \ll n^{\varepsilon}$ one sees that 
\begin{gather} \label{aaa} 
   \lambda(n) \ll_\varepsilon x^\varepsilon, \qquad \qquad (n \ll x).
\end{gather}
The equations \eqref{a} and \eqref{aaa} immediately yield that the number of weak prime pairs less than $x$ is at least of the order $x^{1-\varepsilon}$ for any $\varepsilon>0$, and thus Zhang's proof method gives somewhat stronger results than he states in his paper.  The point here is that we do not only have that the left hand side of the inequality in \eqref{a} is  positive, but also that it is bounded from below by $x^{1-\varepsilon}$. Pintz \cite{Pintz} was the first to come out with a paper including such a result. He furthermore managed to prove\footnote{His results are in fact much more general and much deeper than this, see his paper.}  this result where $x^{-\varepsilon}$ is replaced by a power of $\log x$. 

In this short paper we will consider the question of existence of pairs of primes with bounded gaps in short intervals, such as $[x,x+\Delta(x)]$ for some  function $\Delta(x)=o(x)$. It turns out that in a similar way as in Pintz \cite{Pintz} it is sufficient to analyse the proof of Theorem 1 in Zhang \cite{Zhang} and use his version of Theorem 2 (which seems to be the deepest part of his paper) as is, in order to prove the following result.

\begin{thm} Suppose that  $A \geq  0$ and $\varepsilon>0$.   Then the interval $[x,x+x/(\log x)^A]$ contains at least $(1+o(1))x^{1-\varepsilon}$ pairs of primes $p_1,p_2$ such that  $1<|p_1-p_2|<6 \cdot 10^7$.
\end{thm}

\begin{proof} By using Trudgian's construction \cite{Trudgian} of an admissible set it follows immediately from Lemma 1. \end{proof}
\begin{rem} By some more work, such as using  estimates for the divisor functions (Lemma 8 in Zhang \cite{Zhang}) or similar methods as in Pintz \cite{Pintz} it is possible to improve the lower bound by replacing $x^{-\varepsilon}$ with a power of $\log x$. \end{rem}

\begin{lem}
  Assume that $\Hh$ is an admissible set with at least $k_0=3.5 \cdot 10^6$ elements. Then, given $\varepsilon,A>0$,   the interval $[x,x+x/(\log x)^A]$ contains at least $(1+o(1))x^{1-\varepsilon}$  integers $n$ such that the set $n+\Hh$ contains at least two primes.
\end{lem}
\begin{proof} Lemma 1 follows from using Lemma 2 to estimate the error term in Lemma 3 and Eq. \eqref{aaa}. \end{proof}

As the second Lemma we will state a variant of \cite[Theorem 2]{Zhang} which is a strong version of the Bombieri-Vinogradov theorem.

\begin{lem} Let $\Delta(x)=x (\log x)^{-A}$, and assume that
  \begin{gather} \label{Deltadef} 
    \Delta(\gamma;d,c)=\sum_{\substack{n \equiv c \pmod d \\ x \leq n \leq x+\Delta(x)}} \gamma(n) - \frac 1 {\phi(d)}\sum_{ x \leq n \leq x+\Delta(x)}\gamma(n) \qquad \text{ for} \qquad (d,c)=1.   \\ \intertext{Then for any $B>0$ and $1 \leq i \leq k_0$} \notag
    \sum_{\substack{d<D^2 \\ d | \mathcal P}} \sum_{c \in \mathcal C_i(d)} \abs{\Delta(\theta;d,c)} \ll_{A,B} \Delta(x) (\log x)^{-B}, \\  \intertext{where} \notag \mathcal C_i(d)= \{c:1\leq c \leq d, (c,d)=1, P(c-h_i) \equiv 0 \pmod d \} \, \, \, \text { for } \, \, \,  1 \leq i \leq k_0. \end{gather}
  \end{lem}
\begin{proof} Although stated for short intervals $[x,x+\Delta(x)]$ instead of $[x,2x]$ this follows directly from Zhang's theorem 2. This is because $\Delta(x)$ is just $x$ multiplied by a power of $\log x$ and this power can be accounted for in the error term. \end{proof}

\begin{rem} If we can prove Lemma 2 for $\Delta(x)=x^\theta$ and  some $7/12 <\theta<1$ it would follow that the intervals  $[x,x+x^\theta]$ contains at least $(1+o(1))x^{\theta-\varepsilon}$ weak prime pairs.  This would be of great interest. We have chosen to formulate Lemma 3 such that a proof of Lemma 2 in a short interval will immediately give such a consequence.   Although we have not yet checked all the details of Zhang's proof of Theorem 2,  his approach seems promising also for short intervals. In particular he at many places gets somewhat sharper results than he needs that might be more useful in the short interval case. We also remark that the corresponding short interval version of the classical Bombieri-Vinogradov theorem has been proved by Perelli-Pintz-Salerno \cite{PPS1,PPS2}.
\end{rem}

\begin{lem}
  Assume that $\Hh$ is admissible and $\Hh$ contains at least $k_0=3.5 \cdot 10^6$ elements. Then  there exists some $B>0$ such  that with $\Delta(\theta;d,c)$  defined by \eqref{Deltadef} and $0<\Delta(x) \ll x$ we have that  
 \begin{multline} \notag \sum_{x \leq n \leq x+\Delta(x)} \left( \sum_{i=1}^{k_0} \theta(n+h_i) -\log(3x) \right) (\lambda(n))^2 \geq  (\omega+o(1)) \Delta(x) (\log x)^{k_0+l_0+1}  + \\  +
\mathcal O \left((\log x)^B \sum_{i=1}^{k_0}\sqrt{ \Delta(x)  \sum_{\substack{d<D^2 \\ d | \mathcal P}} \sum_{c \in \mathcal C_i(d)} \abs{\Delta(\theta;d,c)}} \right) + \mathcal O \left(x^{7/12}\right),
\end{multline}
 where we can choose $\omega=\exp(-5 \cdot 10^7)$.
\end{lem}
\begin{proof} We will follow Zhang \cite{Zhang} and just indicate where changes are needed. The essence in the changes is that we would like $n \sim x$ to mean $x \leq n \leq x+\Delta(x)$ instead of  $x \leq n \leq 2x$.  We now proceed  as in Zhang \cite{Zhang}. We write the left hand side of the lemma as
\[
  S_2-\log(3x) S_1,
\] 
where
\begin{gather*}
   S_1=  \sum_{x \leq n \leq x+\Delta(x)} (\lambda(n))^2,   \\ \intertext{and}          S_2=\sum_{x \leq n \leq x+\Delta(x)}(\lambda(n))^2 \,
     \sum_{i=1}^{k_0} \theta(n+h_i). \end{gather*}
We will treat $S_1$ and $S_2$ separately.
\subsection*{Upper bound for $S_1$}

 We follow Zhang \cite{Zhang}, section 4. We treat the inner sum that corresponds to the first displayed formula on p. 16 in the same way. Proceeding in a similar manner, instead of eq. $(4.1)$ in Zhang  we  get  
\begin{gather*} S_1=\mathcal T_1 \Delta(x)+\mathcal  O(D^{2+\varepsilon}). \label{ajj}
\end{gather*} 
The quantity $ \mathcal T_1$ is the same in Zhang, and from his $(4.19)$ we obtain similarly to $(4.20)$ in his paper that
\begin{gather} \label{ac}
  S_1 \leq \frac{(1+\kappa_1+o(1))}{(k_0+2 l_0)!} \binom {2l_0}{l_0} \mathfrak S(\mathcal H) x (\log D)^{k_0+2 l_0}+ \mathcal O (D^{2+\varepsilon}),\end{gather}  for some $\kappa_1< \exp(-1200)$.

\subsection*{Lower bound for $S_2$} 
Again we will change the summation order as in Zhang \cite[p. 22]{Zhang} and we get
\begin{gather} \notag 
 S_2=  k_0  \mathcal T_2 \sum_{x \leq n \leq x+\Delta(x)} \theta(n) + \sum_{i=1}^{k_0} \mathcal O (\mathcal E_i), \\ \intertext{where $\mathcal T_2$ is defined as in his paper and} \label{Edef}
\mathcal{E}_i = \sum_{\substack{d<D^2 \\ d | \mathcal P}} \tau_3(d) \rho_2(d) \sum_{c \in C_i(d)} \left| \Delta(\theta;d,c) \right|, \end{gather} 
where we now keep in mind that $\Delta(\theta;d,c)$ is defined by \eqref{Deltadef}, i.e. summing over $n$ in shorter intervals than in Zhang \cite{Zhang}.
 We may now assume that $\Delta(x) \gg x^{7/12} (\log x)^{-k_0-l_0-1}$, since otherwise the result follows trivially by the error term $\mathcal O(x^{7/12})$. This allows us to use the prime number theorem in short intervals of Heath-Brown \cite{HeathBrown}, instead of the classical prime number theorem in Zhang's treatment. This gives us (compare with \cite[Eq (5.4)]{Zhang}) that
\begin{gather*}
 S_2=  k_0  \mathcal T_2 \Delta(x)(1+o(1)) + \sum_{i=1}^{k_0}  \mathcal O (\mathcal E_i).
\end{gather*} 
Following Zhang\footnote{We are grateful to GH and Denis Chaperon de Lauzi at mathoverflow for explaining precisely how  Zhang used Cauchy's inequality. See  http://mathoverflow.net/questions/132452.} \cite{Zhang} we use the Cauchy inequality on Eq. \eqref{Edef} to estimate the error terms
\begin{gather*}
\mathcal{E}_i \ll \sqrt{ \sum_{\substack{d<D^2 \\ d | \mathcal P}} (\tau_3(d))^2 (\rho_2(d))^2 \sum_{c \in C_i(d)}\left| \Delta(\theta;d,c) \right|} \sqrt{\sum_{\substack{d<D^2 \\ d | \mathcal P}} \sum_{c \in C_i(d)} \left| \Delta(\theta;d,c) \right|}.
 \end{gather*}
 By using the trivial\footnote{this corresponds to each integer being prime in the residue class for the short interval} upper estimate  $|\Delta(\theta;d,c)|\ll \Delta(x)\log x/\phi(d)$ in the first sum, and estimates for sum of divisor functions in short intervals (\cite[Lemma 8]{Zhang}), this gives us the estimate 
\begin{gather} \label{ajaj}
\mathcal{E}_i \ll \mathcal (\log x)^{B} \sqrt{\Delta(x)}  \sqrt{ \sum_{\substack{d<D^2 \\ d | \mathcal P}} \sum_{c \in C_i(d)} \left| \Delta(\theta;d,c) \right|},
\end{gather}
for some positive constant $B>0$. Since $\mathcal T_2$ is the same as in Zhang's paper, it can be estimated by \cite[Eq (5.5)]{Zhang}. Corresponding to Eq (5.6) in Zhang we get the inequality
\begin{gather} \label{ad}
  S_2 \geq \frac{k_0(1-\kappa_2)} {(k_0+2l_0+1)!} \binom{2l_0+2} {l_0+1} \mathfrak S (\Hh) \Delta(x) (\log D)^{k_0+2l_0+1} (1+o(1)).
\end{gather}
 By the inequalities \eqref{ac} and \eqref{ad} we obtain
\begin{multline} \label{ab} S_2 - \log(3x) S_1 \geq \\ \geq \omega \mathfrak S (\Hh) \Delta(x) (\log D)^{k_0+2l_0+1} (1+o(1)) + \mathcal O(D^{2+\epsilon})+\mathcal O\left (\sum_i \mathcal E_i \right), \end{multline}
where
\begin{gather*}
 \omega=\frac 1 {(k_0+2l_0)!} 
\binom{2l_0}{l_0} \left( \frac{2(2 l_0+1)k_0(1-\kappa_2)} {(l_0+1)(k_0+2l_0+1)}-\frac{4(1+\kappa_1)}{1+4 \varpi} \right).
\end{gather*}
While Zhang never calculates $\omega$ and just say it is positive, by the estimates $\kappa_1 <\exp(-1200)$ and $\kappa_2 < 10^8 \exp(-1200)$ from Zhang \cite{Zhang}, and the numerical values of $l_0,k_0,\varpi$ it is easy to use Mathematica/Sage to see that $\omega=3.647 \cdot 10^{-21385285}>\exp(-5 \cdot 10^7).$ The results follows by combining \eqref{ajaj} with \eqref{ab} with the fact that with our choice or $D$ we have $\mathcal O(D^{2+\epsilon})=\mathcal O(x^{7/12})$.
\end{proof}

\end{document}